\documentclass[11pt,a4paper]{article}
\usepackage{authblk}
\usepackage{amsmath}
\usepackage{amssymb}
\usepackage{amsthm}
\usepackage{amsfonts}
\usepackage{enumerate}
\usepackage{pstricks}
\usepackage{pst-plot}
\usepackage{pst-poly}
\usepackage{graphicx}
\usepackage{mathtools} 
\usepackage{url}
\usepackage{booktabs}
\usepackage{multirow}
\usepackage{verbatim}
\usepackage{breakurl,url}
\usepackage[margin=1.15in]{geometry}
\usepackage{mathrsfs}%
\usepackage[title]{appendix}%
\usepackage{xcolor}%
\usepackage{textcomp}%
\usepackage{manyfoot}%
\usepackage{algorithm}%
\usepackage{algorithmicx}%
\usepackage{algpseudocode}%
\usepackage{listings}%
\usepackage{todonotes}

\usepackage[colorlinks=true,allcolors=blue]{hyperref}

\newtheorem{thm}{Theorem}
\newtheorem{lemm}[thm]{Lemma}
\newtheorem{cor}[thm]{Corollary}
\newtheorem{obs}[thm]{Observation}
\newtheorem{conj}[thm]{Conjecture}

\newtheorem{fconj}[thm]{Former conjecture}
\theoremstyle{remark}
\newtheorem{example}{Example}
\newtheorem{remark}{Remark}

\DeclareMathOperator{\ex}{exc} 

\newcommand{\diam}{\text{diam}}
\newcommand{\barG}{\overline{G}} 

\newcommand{\cC}{\mathcal{C}} 
\newcommand{\close}{C_{\barG}} 
\newcommand{\closep}[1]{C_{#1}} 
\newcommand{\isom}{\cong} 

\begin{document}

\title{Structure of betweenness uniform graphs with low values of betweenness centrality}

\author[1]{Babak Ghanbari}
\author[1,2]{David Hartman}
\author[1]{Vít Jelínek}
\author[1,2]{Aneta Pokorná}
\author[1]{Robert Šámal}
\author[3]{Pavel Valtr}

\affil[1]{Computer Science Institute of Charles University, Faculty of Mathematics and Physics, Charles University, Malostransk\'{e} n\'{a}m. 25, Prague 1, 11800, Czech Republic}

\affil[2]{Institute of Computer Science of the Czech Academy of Sciences, Czech Academy of Sciences, Pod Vod\'{a}renskou v\v{e}\v{z}\'{i} 271/2, Prague 8, 18207, Czech Republic}

\affil[3]{Department of Applied Mathematics, Faculty of Mathematics and Physics of Charles University, Malostranské náměstí 25, Prague 1, 11800, Czech Republic}

\date{}
\maketitle

\begin{abstract}
This work deals with undirected graphs that have the same betweenness centrality for each vertex, 
so-called betweenness uniform graphs (or BUGs). The class of these graphs is not trivial and its 
classification is still an open problem. Recently, Gago, Coroničová-Hurajová and Madaras conjectured 
that for every rational $\alpha\ge 3/4$ there exists a BUG having betweenness 
centrality~$\alpha$. We disprove this conjecture, and provide an alternative view of the 
structure of betweenness-uniform graphs from the point of view of their complement. This allows 
us to characterise all the BUGs with betweennes centrality at most 9/10, and show that their 
betweenness centrality is equal to $\frac{\ell}{\ell+1}$ for some integer $\ell\le 9$. We 
conjecture that this characterization extends to all the BUGs with betweenness centrality smaller 
than~1.
\end{abstract}


\section{Introduction}
Complex networks represent a model for many complicated real-world phenomena. Vertices of complex networks 
usually represent some objects of interest, such as actors in social networks, nodes in computer 
networks, proteins in their interaction structure, parts of the brain in connectome networks, etc. 
For other examples, see~\cite{barabasi2016network} and references therein. One of the aims of the study of
complex networks is to identify interesting objects and evaluate their importance. This evaluation 
can be done using so-called centrality measures originating from social networks~\cite{Freeman1979centrality}. These 
measures rate vertices' importance in their role in the corresponding network. Historically, 
the primary approach, called degree centrality, is based on the number of neighbors of the 
corresponding vertex. This local approach has been improved many times leading to more elaborate 
measures accounting for broader neighborhoods, such as eigenvector 
centrality~\cite{Bonacich1987power}. As an alternative to the measures based on the size and importance of the neighborhood of vertices, centrality can also be defined as a mediator of communication in the network. Communication 
is expressed as the amount of shortest paths passing through the given vertex. This 
approach is formalized by the notion of betweenness centrality, which counts the overall fraction of shortest paths connecting pairs of vertices passing through the given 
vertex~\cite{freeman1977centrality-measures}. 

To use any centrality measure, it is beneficial to know its typical behavior. Motivated by degree centrality and degree distribution studies resulting in concepts such as scale-freeness, one can ask about the typical behavior of other centrality measures, such as betweenness centrality. A natural starting point is the limiting case when betweenness centrality is the same for all vertices. We will call such structures betweenness-uniform graphs (shortly BUGs). Note that the value of the betweenness centrality of any vertex in a BUG is thus equal to the average betweenness centrality in the graph. 

Characterizing the structure of BUGs is not a trivial task. The class of BUGs includes all vertex-transitive graphs, but not all BUGs are vertex-transitive; indeed, some BUGs are not even regular. We can go even further. Based on the correspondence with strongly regular graphs, we can construct a BUG for any automorphism group~\cite{Gago2013Onbetweenness}. There are several approaches to characterize the betweenness-uniform graphs. One possibility is to find bounds for different characteristics, such as vertex connectivity; see a recent result of Hartman, Pokorná and Valtr~\cite{hartman2024connectivity,hartman2021connectivity}. Another approach is to identify all the possible rational numbers that can occur as the betweenness value of a~BUG. It has been conjectured that all rational numbers greater than $\frac34$ have this property. 

\begin{fconj}[Gago, Coroničová-Hurajová, Madaras 2013 \cite{Gago2013Onbetweenness}]\label{conj}
For any rational value $\alpha$ in the interval $(\frac34, \infty)$, there exists a 
betweenness-uniform graph with betweenness\footnote{This conjecture is mentioned 
in the text~\cite{Gago2013Onbetweenness} after Lemma 2.3. Note that the definition of betweenness 
centrality in the mentioned article considers oriented pairs instead of unoriented ones and thus 
the corresponding values are twice as large, i.e. the interval is $(3/2, \infty)$.} value~$\alpha$.
\end{fconj}

It has already been shown that the only BUGs with the betweenness value of zero are complete graphs 
and that there are no BUGs with betweenness in the interval $(0, 
\frac{1}{2})$~\cite{Gago2013Onbetweenness}. In this work, we examine the structure of BUGs with 
betweenness below one, which we call low-betweenness BUGs. These graphs can be conveniently 
described by the structure of their complements, which we abbreviate as coBUGs. We show that the 
complement of a low-betweenness BUG needs to have fewer edges than vertices, and in fact the 
complement must be disconnected, implying that a low-betweenness BUG has diameter at most~2.

An easy way to construct a low betweenness BUG is to fix integers $k\ge 2$ and $\ell\ge 0$ and 
consider the complement of the graph obtained as a disjoint union of $k$ copies of the star 
$K_{1,\ell}$. This yields a BUG with betweenness value $\frac{\ell}{\ell+1}$. The question remains 
whether there are other examples of low-betweenness BUGs, or even other possible values of 
betweenness smaller than~1. We conjecture that this is not the case. 
\begin{conj}\label{our-conj}
    If $G$ is a low-betweenness BUG, then its complement $\barG$ is a disjoint union of stars of the same size.
\end{conj}

The study of the structure of coBUGs can be viewed from the perspective of combining various types 
of graphs as coBUG components, e.g. Corollary~\ref{cor-no-stars-and-cycles} shows that a star cannot 
be combined with a cycle. For a subinterval $(0, 9/10\rangle$, we show that stars $K_{1,\ell}$ for 
$\ell \in \{1, \dots, 8\}$ cannot be combined with other graphs as connectivity components in a 
coBUG that correspond to a low-betweenness BUG. This implies the next theorem, which is our main 
result.

\begin{thm}\label{thm-main}
Let $G$ be a BUG with betweenness $\alpha\le 9/10$. Then $\alpha=\ell/(\ell+1)$ for some integer $\ell\ge 0$, and the complement of $G$ is a graph whose every connected component is isomorphic to~$K_{1,\ell}$. 
\end{thm}

This disproves Conjecture~\ref{conj}, and provides partial support for Conjecture~\ref{our-conj} by showing that any counterexample for the conjecture must have betweenness within the interval $(9/10,1)$. As further support for Conjecture~\ref{our-conj}, we also show 
that the complement of any counterexample to the conjecture must contain a component which is a star on  at least 10 vertices together with a component of a very irregular structure; see Subsection~\ref{ssec-mt} and Theorem~\ref{thm-mt} for the precise statement.

It appears that the structure of BUGs with betweenness less than one is relatively simple to describe in 
terms of their complements, and in particular, Conjecture~\ref{our-conj} implies that for a given 
$\alpha<1$ and a given number of vertices~$n$, there is up to isomorphism only one BUG on $n$ 
vertices with betweenness~$\alpha$. The picture becomes much more complicated as soon as the value 
$\alpha=1$ is reached. For instance, it can be easily verified that any graph whose complement is a 
disjoint union of cycles of length at least 4 (not necessarily all of the same length) is a BUG 
with betweenness 1~(see Example~\ref{ex-cycles}). But these are not the only BUGs of betweenness~1; 
in fact, for any complete multipartite graph $H$ with more edges than vertices, we can construct a 
BUG $G$ of betweenness 1 whose complement $\barG$ contains $H$ as a connected component, with all 
the other components of $\barG$ being stars (see Example~\ref{ex-mt}). Moreover, we also 
construct a sequence of BUGs whose betweenness values converge to 1 from above 
(Example~\ref{ex-above}). 

\section{Preliminaries}
At the beginning of this section, we introduce all the definitions and notation together with basic notions, ideas, and observations leading to deeper insight into the problem and into the motivation behind the key concepts. These basic insights will be also frequently used in the following sections, which contain our main results.

\subsection{Definitions and notation}
Throughout this article, a network is represented by an unoriented graph $G$ with a set of vertices 
$V(G)$ and a set of edges $E(G) \subseteq \binom{V(G)}{2}$. For simplicity, in most cases, 
we will denote the edge between the vertices $x,y \in V(G)$ simply as $xy \in E(G)$ instead of the 
more accurate $\{x,y\}\in E(G)$. The {\it complement} $\barG$ of a graph $G$ is the graph with vertex 
set $V(G)$ and edge set $\binom{V(G)}{2} \setminus E(G)$. 
The symbol $\isom$ denotes graph isomorphism. 
The number of vertices of the currently discussed graph is denoted by~$n$.
The set of vertices adjacent to a vertex $v$ is denoted by $N(v)$, and $N[v]$ denotes the set $N(v) \cup \{v\}$. The {\it degree} of a vertex $v$, denoted by $\deg(v)$, is equal to~$|N(v)|$.
The {\it maximum degree} of $G$, denoted $\Delta(G)$, is $\max_{v \in V(G)} \deg(v)$.

The {\it betweenness centrality} (or just \emph{betweenness)}  
of a vertex $x \in V(G)$ is defined as\footnote{Note that some authors use an alternative 
definition of betweenness that considers ordered pairs. In that case, the summation is across 
all $(v,w) \in V(G) \times V(G)$ such that $u \neq v$. The resulting value of betweenness is two 
times larger than the value corresponding to the definition we use.} 
\begin{equation*}
    B_G(x) = \sum_{\{v,w\} \subseteq V(G)} \frac{\sigma_{vw}(x)}{\sigma_{vw}},
\end{equation*}
where $\sigma_{vw}$ denotes the number of shortest paths between $v$ and $w$ and $\sigma_{vw}(x)$ 
is the number of such paths containing $x$ as an internal 
vertex~\cite{freeman1977centrality-measures}. In particular, $\sigma_{vw}(x) = 0$ whenever $x = v$ 
or $x = w$. If the graph $G$ is clear from context, we write $B(x)$ instead of~$B_G(x)$.
Betweenness centrality can be also defined for the whole graph $G$ by taking the average betweenness of its vertices, i.e.,
$$B(G) = \frac{\sum_{x \in V(G)} B(x)}{|V(G)|}.$$
We say that a graph $G$ is a \emph{betweenness-uniform graph} (or shortly a \emph{BUG}), if all its 
vertices have the same betweenness, that is, $B(x) = B(y)$ for all $x,y \in V(G)$. Note that in such 
a case, the betweenness of the whole graph is identical to the betweenness of any vertex.
The complement $\barG$ of a BUG $G$ will be referred to as a {\it coBUG}.

In this text, we will focus on betweenness-uniform graphs with values of betweenness smaller 
than~1. We call such graphs \emph{low-betweenness BUGs}. The complement of a low-betweenness BUG is 
called a \emph{low-betweenness coBUG}.

\subsection{Bounds for betweenness via edge-vertex ratio}

When dealing with BUGs, especially BUGs with low betweenness, it is often more natural to focus
on the properties of the corresponding coBUG rather than the BUG itself. This is illustrated by the 
next lemma, which we will frequently use throughout the rest of the paper.
\begin{lemm}\label{lem-density}
For a connected graph $G$ on $n$ vertices, we have $B(G) \ge \frac{|E(\barG)|}{n}$, with 
equality if and only if $G$ has diameter at most~2.
\end{lemm}
\begin{proof}
Let $G=(V,E)$ be a graph, let $\barG=(V,\overline{E})$ be its complement, and let 
$n:=|V|$. For two distinct vertices $x,y\in V$, let $d_G(x,y)$ be the distance of $x$ and $y$ in 
$G$, i.e., the number of edges on the shortest $ xy$ path in~$G$. Recall that $\sigma_{xy}$ is the 
number of shortest $xy$-paths, and $\sigma_{xy}(u)$ is the number of those shortest $xy$-paths that 
contain $u$ as an internal vertex. Notice that for any two distinct vertices $x,y\in V$, 
\[
\sum_{u\in V} \sigma_{xy}(u)=(d_G(x,y)-1)\sigma_{xy},
\]
since both sides of the equation count the number of pairs $(u,P)$ where $P$ is a shortest 
$xy$-path and $u$ is an internal vertex of~$P$. We then have
\begin{align*}
B(G)&=\frac{1}{n}\sum_{u\in V} B(u)\\
&=\frac{1}{n}\sum_{u\in V} \sum_{\{x,y\}\in\binom{V}{2}}\frac{\sigma_{xy}(u)}{\sigma_{xy}}\\
&=\frac{1}{n}\sum_{\{x,y\}\in\binom{V}{2}} \sum_{u\in V}\frac{\sigma_{xy}(u)}{\sigma_{xy}}\\
&=\frac{1}{n}\sum_{\{x,y\}\in\binom{V}{2}}(d_G(x,y)-1)\\
&=\frac{1}{n}\sum_{\{x,y\}\in \overline{E}}(d_G(x,y)-1)\\
&\ge \frac{|\overline{E}|}{n},
\end{align*}
where the last inequality holds with equality if and only if any two distinct nonadjacent vertices 
of $G$ have distance two, i.e.,  if and only if $G$ has diameter at most two.
\end{proof}

For a graph with $m$ edges and $n$ vertices, the ratio $\frac{m}{n}$ is referred to as the 
\emph{edge-vertex ratio}. It is precisely half of the average degree of the graph. 
Lemma~\ref{lem-density} then shows that the betweenness of a BUG $G$ is lower-bounded by the 
edge-vertex ratio of its complement, with equality for BUGs of diameter~two.

\begin{cor}\label{cor-coBUG-density}
    If $G$ is a low-betweenness BUG, then $\barG$ has fewer edges than vertices.
\end{cor}

This corollary means that for a low-betweenness BUG, the only possible corresponding coBUG is either a tree or a disconnected graph. The informal message of this section is that low-betweenness BUGs are always relatively dense.



\subsection{Low-betweenness BUGs with diameter at least three}\label{ssec-diam3}
In this section, we show that any low-betweenness BUG is of diameter at most two, thus 
proving Theorem~\ref{thm-diam3}.
This result, together with Lemma~\ref{lem-density}, simplifies our study of low-betweenness BUGs. 
The key part is realizing that complements of graphs with diameter at least three have a specific form.

\begin{lemm}\label{lem-doublestar}
A connected graph $G$ satisfies $\diam(G) \geq 3$ if and only if $\barG$ is connected and 
contains a spanning tree which is a double star.
\end{lemm}

\begin{proof}
We prove a stronger version of this lemma saying that two vertices $a$ and $b$ have a mutual distance of at least three in $G$ if and only if $\barG$ contains a spanning tree which is a double star with the two centers $a$ and $b$.

First, suppose that $a$ and $b$ have mutual distance at least three in $G$. Then $ab$ is not an edge in $G$ and thus it is an edge of $\barG$. Similarly, for any vertex $c\ne a,b$, the graph $G$ does not contain the path $acb$ and therefore $\barG$ contains at least one of the edges $ac$ and $cb$.
It follows that $\barG$ contains a spanning double star with the two centers $a$ and $b$.

Suppose now that $\barG$ contains a spanning double star with the two centers $a$ and $b$.
Then $G$ contains neither an edge $ab$ nor any path $acb$, since in the latter case $c$ would not be connected by an edge to any of the two centers of the spanning double star in $\barG$, which is impossible. It follows that $a$ and $b$ have  distance at least three in $G$.
This completes the proof.
\end{proof}

The spanning double star is inevitable in the complement of any graph of diameter at least three. This property allows us to show that any coBUG containing it must have the edge-vertex ratio at least one. Using this corollary together with Lemma~\ref{lem-density} we can show that there are no low-betweenness BUGs of diameter at least~three.

\begin{thm}\label{thm-diam3}
If $G$ is a connected BUG of diameter at least 3, then $B(G)\ge 1$.
\end{thm}
\begin{proof}
    Suppose for contradiction that $G=(V,E)$ is a BUG of diameter at least 3 with $B(G)<1$. By Lemma~\ref{lem-doublestar}, the complement $\barG=(V,\overline{E})$ is connected and contains a spanning double star. By Lemma~\ref{lem-density}, we know that $\frac{|\overline{E}|}{|V|}\le B(G)<1$, and in particular, $|\overline{E}|<|V|$. It follows that $\barG$ is a tree, hence $\barG$ itself is a double star. Let $a$ and $b$ be the two centers of the double star~$\barG$. We may easily check that $G$ does not contain any shortest path whose internal vertex is $a$ or~$b$. Consequently, $B(a)=B(b)=0$, and since $G$ is a BUG, we have $B(G)=0$. However, the only connected BUGs of betweenness equal to 0 are  complete graphs, whose diameter is equal to~1. This contradiction completes the proof.
\end{proof}

\subsection{BUGs with diameter two}\label{ssec-diam2}

Lemma \ref{lem-density} implies that if a BUG $G$ has low betweenness, then its complement 
$\barG$ has fewer edges than vertices. This is our motivation for defining the {\it edge excess} of 
a graph $H$ as $\ex(H) := |E(H)| - |V(H)|$. Thus, complements of low-betweenness BUGs have negative 
edge excess.

This shows, in particular, that complements of low-betweenness BUGs are generally rather sparse, 
and therefore the BUGs themselves rather dense. It is therefore easier and more natural, when 
dealing with low-betweenness BUGs, to describe the structure of the corresponding coBUGs.

Theorem~\ref{thm-diam3} shows that in our search for low-betweenness BUGs, we may restrict our 
attention to graphs of diameter at most~2. Graphs of diameter 1 are trivial: they are the complete graphs, which as we know are betweenness-uniform with betweenness equal to~0. Henceforth, we will therefore focus on graphs of diameter~2.

Fortunately, it turns out that for a graph $G$ of diameter 2, betweenness uniformity can be easily 
characterised in terms of the complement~$\barG$. We now introduce the required terminology and 
state the characterisation.

Let $H=(V,E)$ be a graph. We say that
a vertex $v\in V$ is {\it close} to an edge $e=xy \in E$, if $v$ is adjacent to at least one 
endpoint of $e$; in particular, the two endpoints $x$ and $y$ of $e$ are close to~$e$. Let 
$\closep{H}(e)$ be the set of vertices close to the edge $e$, and let $\closep{H}(v)$ be the set of 
edges that are close to the vertex~$v$. Notice that $\closep{H}(e)$ is equal to $N_{H}(x) \cup 
N_{H}(y)$, and that $\closep{H}(v)$ contains precisely those edges that are incident to at least 
one vertex in $N_H(v)$,

\begin{example}\label{ex-closeness}
Let $H\isom C_k$ be a graph isomorphic to a cycle of length $k \geq 4$. For any edge $e$ of 
the cycle, we can compute $|\closep{H}(e)| = 4$, since we have two end vertices of the edge and 
each of them has one neighbor. Similarly, for any vertex $v$, $|\closep{H}(v)| = 4$, since each of 
the two neighbors has one edge back to vertex $v$ and one to the next neighbor.
\end{example}

Let $H=(V,E)$ be a graph on $n$ vertices. For an edge $e\in E$, we define its \emph{weight} as 
\[
 w_H(e)=\frac{1}{n-|\closep{H}(e)|}.
\]
In case $\closep{H}(e)=n$, i.e., if the edge $e$ is close to all the vertices, we leave $w_H(e)$ 
undefined. Note that this only happens when $H$ contains as a subgraph a double-star with central 
edge $e$, and in view of Lemma~\ref{lem-doublestar}, the complement of $H$ then has diameter at 
least~$3$. Such situation will not occur in our applications.

We further define the \emph{co-betweenness} of a vertex $x$ as 
\[
 coB_H(x)=\sum_{e \in \closep{H}(x)} w_H(e),
\]
i.e., the co-betweenness of a vertex is the total weight of the edges close to it.
If the graph $H$ being considered is clear from the context, we write $\closep{}(e)$, $w(e)$ and 
$coB(x)$, instead of $\closep{H}(e)$, $w_H(e)$ and $coB_H(x)$. In all our applications, the role of 
$H$ will be played by the complement $\barG$ of a (potential) BUG~$G$. The relevance of these 
notions is explained by the next lemma.

\begin{lemm}\label{lem-weight}
 Let $G=(V,E)$ be a graph of diameter at most 2, and let $x$ be a vertex of~$G$. Let $\barG$ be the
complement of~$G$. The betweenness $B_G(x)$ of the vertex $x$ satisfies
\begin{equation}
 B_G(x)=\left(\sum_{e \in E(\barG)} w_{\barG}(e)\right) - coB_{\barG}(x).
\end{equation}\label{eq-weight}
\end{lemm}
\begin{proof}
Since $G$ has diameter at most 2, for any pair of distinct vertices $e=\{v,w\}$, there can be at 
most one shortest path from $v$ to $w$ passing through the vertex $x$, and such a path exists 
if and only if the following two conditions hold:
\begin{itemize}
 \item[(a)] $e$ is not an edge of $G$, i.e., $e\in E(\barG)$, and
 \item[(b)] both edges $xv$ and $xw$ are in $G$, or equivalently, the vertex $x$ 
is not close to the edge $e$ in~$\barG$. 
\end{itemize}
It follows that for a pair of distinct vertices $e=\{v,w\}$, the number of shortest 
paths from $v$ to $w$ in $G$ is equal to 1 if $e$ is an edge of $G$, and is equal to 
$n-|\closep{\barG}(e)|$ if $e$ is an edge of~$\barG$.

We may then compute the betweenness of $x$ as follows:
\begin{align*}
 B_G(x)&=\sum_{\{v,w\} \subseteq V} \frac{\sigma_{vw}(x)}{\sigma_{vw}}\\
 &=\sum_{\{v,w\} \in E(\barG)} \frac{\sigma_{vw}(x)}{\sigma_{vw}}\\
&=\sum_{e=\{v,w\} \in E(\barG)} \frac{\sigma_{vw}(x)}{n-|\closep{\barG}(e)|}\\
&=\sum_{e=\{v,w\} \in E(\barG)} w_{\barG}(e)\sigma_{vw}(x)\\
&=\sum_{e\in E(\barG)\setminus\closep{\barG}(x)} w_{\barG}(e) &\text{by conditions (a) and (b) above}\\
&=\left(\sum_{e \in E(\barG)} w_{\barG}(e)\right) - \sum_{e \in \closep{\barG}(x)} w_{\barG}(e)\\
&=\left(\sum_{e \in E(\barG)} w_{\barG}(e)\right) - coB_{\barG}(x)
\end{align*}
as claimed.
\end{proof}

\begin{cor}\label{cor-weight}
A graph $G$ of diameter at most 2 is betweenness-uniform if and only if in its complement $\barG$, 
all the vertices have the same co-betweenness.
\end{cor}
Note that to determine the co-betweenness of a vertex $x$, we only need to know the weights of the 
edges close to $x$, and to determine the weight of an edge $e$, we only need to know the total 
number of vertices in the graph and the number of vertices close to~$e$. Thus, to determine the 
betweenness-uniformity of a graph $G$ by inspecting its complement $\barG$, we only need to look 
at the local structure in $\barG$, with the only `global' property needed being the total number of 
vertices. This of course only works when $G$ has diameter at most 2; however, in the examples we 
encounter, the complement $\barG$ will typically be a disconnected graph, which automatically 
guarantees that $G$ has diameter at most~2.  

Corollary~\ref{cor-weight} allows us to restrict the possible structure of graphs that may appear 
as connected components of a disconnected coBUG.

\begin{lemm}\label{lem-close-inclusion}
Let $H$ be a graph with two vertices $x,y$ such that $\closep{H}(x)\subsetneq\closep{H}(y)$. Then 
$H$ cannot be a connected component of any disconnected coBUG.
\end{lemm}
\begin{proof}
Suppose $\barG$ is a disconnected coBUG with a component~$H$. The strict inclusion 
$\closep{H}(x)\subsetneq\closep{H}(y)$ implies that $coB_{\barG}(x)<coB_{\barG}(y)$, which is 
impossible in a coBUG.
\end{proof}

\begin{example}\label{ex-cycles}
Let $G$ be a graph on $n$ vertices. Suppose that its complement $\barG$ is a disconnected graph 
whose every component is a cycle of length at least~4. Then $G$ is a BUG of betweenness~1. Indeed, 
in $\barG$, each edge is close to four vertices, and hence has weight $\frac{1}{n-4}$, and each 
vertex is close to four edges, and hence has co-betweenness $\frac{4}{n-4}$. Since $\barG$ is 
disconnected, $G$ has diameter at most~2, and since all the 
vertices of $\barG$ have the same co-betweenness, $G$ is a BUG by Corollary~\ref{cor-weight}. To 
see that $G$ has betweenness 1, we may either use Lemma~\ref{lem-weight}, or better yet, 
Lemma~\ref{lem-density}.
\end{example}

An easy way to obtain graphs where all the vertices have the same co-betweenness is to ensure that 
each vertex is close to the same number of edges, and each edge is close to the same number of 
vertices. A graph $H$ is \emph{$(m,t)$-uniform} if for any vertex $v \in V(H)$, $|\closep{H}(v)| = 
m$, and for any edge $e \in E(H)$, $|\closep{H}(e)| = t$. Clearly, a graph is $(m,t)$-uniform if 
and only if each of its components is $(m,t)$-uniform. We say that a graph is \emph{uniform}, if it 
is $(m,t)$-uniform for some values of $m$ and~$t$.

\begin{obs}\label{obs-coB-mt-unif}
    In an $(m,t)$-uniform graph on $n$ vertices, every vertex has co-betweenness $\frac{m}{n-t}$. 
In particular, every disconnected uniform graph is a coBUG by Corollary~\ref{cor-weight}.
\end{obs}

For future use, we remark that $(m,t)$-uniformity determines the edge-vertex ratio:

\begin{lemm}\label{lem-unifdens}
Any $(m,t)$-uniform graph $H$ satisfies $\frac{m}{t}=\frac{|E(H)|}{|V(H)|}$.
\end{lemm}
\begin{proof}
Count in two ways the number of pairs $(v,e)\in V(H) \times E(H)$ such that $e\in \closep{H}(v)$ to 
get $|V(H)|m=|E(H)|t$.
\end{proof}

A special case of uniform graphs are graphs in which every edge is close to every vertex. 
Note that such graphs are precisely the complete multipartite graphs:

\begin{obs}\label{obs-multipar}
 For a graph $H=(V,E)$ the following properties are equivalent:
 \begin{itemize}
  \item $H$ is $(m,t)$-uniform with $m=|E|$ and $t=|V|$,
  \item every edge of $H$ is close to every vertex of $H$,
  \item $H$ has no induced subgraph isomorphic to $K_2+K_1$ (i.e., the disjoint union of $K_2$ and 
$K_1$), 
\item $H$ is complete multipartite (i.e., its complement is a disjoint union of cliques).
 \end{itemize}
\end{obs}

Note that a connected complete multipartite graph with at least two vertices cannot be a coBUG, 
despite being uniform, since its complement is disconnected and therefore not a BUG; 
however, complete multipartite graphs, and connected uniform graphs in general, are natural 
candidates for connected components of coBUGs. We will later see, however, that to obtain new 
examples of low-betweenness BUGs, we must consider coBUGs with components that are not 
uniform.

\section{Low-betweenness BUGs}

We will now focus on the specific properties of BUGs with betweenness less than~1. We know from 
Theorem~\ref{thm-diam3} that such BUGs have diameter at most 2, and therefore the results of 
Subsection~\ref{ssec-diam2}, in particular Lemma~\ref{lem-weight} and Corollary~\ref{cor-weight}, 
apply to them. 

By Lemma~\ref{lem-density}, the complement of a low-betweenness BUG has fewer edges than vertices, 
and in particular, at least one of its components must be a tree. We show that the only trees that 
can occur are stars.

\begin{lemm}\label{lem-stars}
 Let $\barG$ be a low-betweenness coBUG on $n$~vertices. Then 
\begin{itemize}
  \item at least one connected component of $\barG$ is a tree,
  \item any component of $\barG$ containing a vertex of degree at most 1 (in particular any tree 
component) is isomorphic to a star $K_{1,\ell}$ for some $\ell$, and
  \item all the star components of $\barG$ are pairwise isomorphic, i.e., they are all stars of the 
same size.
\end{itemize}
\end{lemm}
\begin{proof}
 Since $\barG$ has fewer edges than vertices, at least one of its components must be a tree. In 
particular, at least one component of $\barG$ contains a vertex of degre at most 1. Let 
$H$ be such a component, and let $x$ be a vertex of degree at most 1 in~$H$. If $x$ has degree 0, 
then $coB(x)=0$, and by Corollary~\ref{cor-weight}, all the vertices of $\barG$ have co-betweenness 
0, which implies that $\barG$ has no edges and the lemma holds with $\ell=0$. 

Suppose now that $x$ has degree 1, and let $y$ be 
the vertex adjacent to~$x$. Note that the set $\close(x)$ contains precisely the edges incident 
to~$y$. If $H$ is not a star, then $H$ contains edges not belonging to $\close(x)$, and at least 
one such edge is close to~$y$. It follows that $\close(x)\subsetneq\close(y)$, which is impossible 
by Lemma~\ref{lem-close-inclusion}. Therefore $H$ must be a star, i.e., isomorphic to $K_{1,\ell}$ 
for some~$\ell$.

It follows that any vertex in $H$ has co-betweenness $\frac{\ell}{n-(\ell+1)}$, and therefore all 
the vertices of $\barG$ have this co-betweenness as well. Suppose that $\barG$ 
contains another tree component $H'$, necessarily isomorphic to a star $K_{1,m}$ for some~$m$. Then 
$\frac{\ell}{n-(\ell+1)}=\frac{m}{n-(m+1)}$, since all the vertices in $\barG$ have the same 
co-betweenness, and this implies $\ell=m$. We conclude that all the star components of $\barG$ are 
pairwise isomorphic.
\end{proof}

Observe that for any $k\ge 2$ and $\ell\ge 0$, the graph obtained by taking $k$ disjoint copies of 
the star $K_{1,\ell}$ is a coBUG, whose complement is a BUG of betweenness $\frac{\ell}{\ell+1}$. 
Our Conjecture~\ref{our-conj} states that this is the only way to construct a low-betweenness BUG. 
Let us call a BUG $G$ \emph{exotic} if it is a low-betweenness BUG whose complement is not a 
disjoint union of stars of the same size; thus, Conjecture~\ref{our-conj} equivalently states that 
there are no exotic BUGs. The complement of an exotic BUG is referred to as \emph{an exotic 
coBUG}.

In the two following subsections, we prove two results which represent partial progress 
towards Conjecture \ref{our-conj}. First, in Subsection~\ref{ssec-mt}, we show that no exotic coBUG 
may contain a uniform component other than a star. Thus, any construction of an exotic coBUG must 
combine some number of star components of the same size (due to Lemma~\ref{lem-stars}), together 
with at least one component which is not uniform.

Next, in Subsection~\ref{ssec-small}, we proceed towards Conjecture \ref{our-conj} by showing that 
no exotic coBUG may contain a star component $K_{1,\ell}$ for $\ell\le 8$. This implies that any
component in a potential exotic coBUG is either a star with at least 10 vertices, or a non-uniform 
connected graph which is not a tree, and therefore has at least as many edges as vertices. Overall, 
such an exotic coBUG has edge-vertex ratio greater than 9/10. By Lemma~\ref{lem-density}, this 
means that any exotic BUG has betweenness more than 9/10. This line of reasoning is an outline of 
the proof of Theorem~\ref{thm-main}, which is presented in more detail at the end of the section.

The results of Subsection~\ref{ssec-small} are proven by means of a case analysis, where we 
separately consider stars $K_{1,\ell}$ for values of $\ell=0,\dotsc,8$, and show that they cannot 
be combined with non-uniform graphs to form a coBUG. There is no reason why this approach could not 
be pushed further to values of $\ell$ greater than 8, thus improving the bound 9/10 from 
Theorem~\ref{thm-main}. However, the arguments become more tedious and lengthy as $\ell$ increases, 
so it might be perhaps more worthwhile to look for a wholly different approach, with the aim of 
solving the full Conjecture~\ref{our-conj}.

\subsection{No exotic coBUGs with uniform components}\label{ssec-mt}

Our next goal is to show that there can be no exotic coBUG with a non-star uniform component. We 
first derive a collection of equalities and inequalities restricting the possible ranges of 
parameters $n,m,t,\ell$ for which one may have a coBUG on $n$ vertices combining a component 
$K_{1,\ell}$ with a non-star $(m,t)$-uniform component.

\begin{lemm}\label{lem-starH}
Let $\barG$ be a coBUG on $n$ vertices with a component $K_{1,\ell}$ for some $\ell$, and another 
component $H$ not isomorphic to $K_{1,\ell}$. Suppose $H$ is $(m,t)$-uniform. Then the following 
holds:
\begin{enumerate}
\item $n=\ell+1 +\ell\cdot \frac{\ell+1 - t}{m-\ell}$,
\item $\ell\cdot\frac{\ell+1 - t}{m-\ell}\ge t$,
\item $m>\ell$ and $t<\ell+1$, and
\item $\ell(\ell+1)\ge mt$.
\end{enumerate}
In particular, if $H$ is a complete multipartite graph, which implies that $m=|E(H)|$ and 
$t=|V(H)|$, then the third item above says that $K_{1,\ell}$ has fewer edges but more 
vertices than $H$.
\end{lemm}
\begin{proof}
Recall that co-betweenness of a vertex is $coB(v) = \sum_{e \in \close(v)} w(e)$ and that in 
disconnected coBUGs, the co-betweenness of any two vertices must be the same by 
Corollary~\ref{cor-weight}.
In the component of $\barG$ isomorphic to a star $K_{1, \ell}$, each vertex has co-betweenness 
$\frac{\ell}{n-(\ell+1)}$, while in the component~$H$, each vertex has co-betweenness 
$\frac{m}{n-t}$ by Observation~\ref{obs-coB-mt-unif}. We thus get the identity
\begin{equation}
\frac{\ell}{n-(\ell+1)}=\frac{m}{n-t}.
\end{equation}\label{eq-lnmt}

We first claim that $m\neq\ell$. Indeed, if we had $m=\ell$, then the above equality would imply 
$t=\ell+1$, and by Lemma~\ref{lem-unifdens}, this would mean that $H$ has fewer edges than vertices 
and is therefore a tree. But by Lemma~\ref{lem-stars}, all the tree components of $\barG$ must be 
isomorphic to $K_{1,\ell}$, contradicting our assumptions about~$H$.

Knowing that $m\neq \ell$, we can manipulate the identity \eqref{eq-lnmt} further:
\begin{align*}
&&\frac{\ell}{n-(\ell+1)}&=\frac{m}{n-t}\\
\iff&&\ell(n-t)&=m(n-(\ell+1))\\
\iff&&n(\ell-m)&=\ell t- m(\ell+1)\\
\iff&&n&=\frac{m(\ell+1)-\ell t}{m-\ell}=\frac{m(\ell+1)-\ell(\ell+1)+\ell(\ell+1)-\ell t}{m-\ell}\\
\iff&&n&=\ell+1+\ell\cdot \frac{\ell+1 - t}{m-\ell}.
\end{align*}
This proves the first claim of the lemma. Since $\barG$ has at least one component isomorphic to 
$K_{1,\ell}$ and another isomorphic to $H$, we know that $n\ge \ell+1+|V(H)|$. Since $H$ has at 
least $t$ vertices, this yields $n=\ell+1 +\ell\cdot \frac{\ell+1 - t}{m-\ell}\ge \ell+1+t$, from 
which we may deduce the second claim. 

From the second claim, we know in particular that $\ell\cdot \frac{\ell+1 - t}{m-\ell}>0$. This 
means that we have either $(\ell+1)<t$ and $m<\ell$, or $(\ell+1)>t$ and $m>\ell$. But the first 
alternative would again mean that $t\ge m+2$, i.e. $H$ has more vertices than edges, so $H$ is a 
tree, but this is again impossible by Lemma~\ref{lem-stars}. So the second alternative must be 
correct, proving the third claim. 

The final claim is deduced from the second and third one:
\begin{align*}
&&\ell\cdot\frac{\ell+1 - t}{m-\ell}&\ge t\\
\iff&&\ell\cdot\frac{\ell+1 - t}{m-\ell}- t&\ge 0\\
\iff&&\frac{\ell(\ell+1) -\ell t - t(m-\ell)}{m-\ell}&\ge 0\\
\iff&&\frac{\ell(\ell+1) - mt}{m-\ell}&\ge 0\\
\iff&&\ell(\ell+1) - mt&\ge 0\\
\iff&&\ell(\ell+1)&\ge mt.\\
\end{align*}
\end{proof}

As an example of a simple application of the previous lemma, let us derive the following useful 
consequence.

\begin{cor}\label{cor-no-stars-and-cycles}
If a coBUG contains a component isomorphic to a star $K_{1,\ell}$ then it cannot contain a 
component isomorphic to a cycle. This is because a cycle is either $(3,3)$-uniform (for $C_3$) or 
$(4,4)$-uniform (for $C_k$ with $k\ge 4$), and the third part of Lemma~\ref{lem-starH} would imply 
either $\ell<3<\ell+1$ or $\ell<4<\ell+1$, which is impossible.
\end{cor}

Now we can proceed to the main result of this subsection.

\begin{thm}\label{thm-mt}
An exotic coBUG cannot contain any uniform connected component other than a star.
\end{thm}


\begin{proof}
For a contradiction, let $\barG$ be an exotic coBUG on $n$ vertices that contains an $(m,t)$-uniform 
connected component $H$ which is not a star.
By Lemma~\ref{lem-stars}, $\barG$ contains a component isomorphic to a star $K_{1,\ell}$ for some 
$\ell$, and all the tree components of $\barG$ are isomorphic to the same star $K_{1,\ell}$.

From Lemma~\ref{lem-starH}, we know that
\begin{equation}
n=\ell+1+\ell\cdot \frac{\ell+1 - t}{m-\ell}. \label{eq-n}
\end{equation}
The same lemma also shows that $m>\ell$ and $t<\ell+1$, and in particular $m>t$.

Let $e$ denote the number of edges of $\barG$. Recall that the excess $\ex(\barG)$ of $\barG$ is 
defined as $e-n$. To reach contradiction, we will now show that $\barG$'s excess is non-negative, 
which by Lemma~\ref{lem-density} implies that the BUG $G$ has betweenness at least 1, and therefore 
is not an exotic BUG. 

We will estimate the excess of $\barG$ componentwise. Let 
$\cC$ be the set of connected components of $\barG$, let $c_H\in\cC$ be a connected component 
isomorphic to $H$, $\cC_1\subseteq\cC$ be the set of the components of $\barG$ isomorphic to 
$K_{1,\ell}$, and let $\cC_2$ be the set $\cC\setminus(\{c_H\}\cup\cC_1)$ of all the remaining 
components. For a component $c\in\cC$, let $e(c)$ and $n(c)$ denote the number of edges and the 
number of vertices of $c$, respectively.

Note that $|\cC_1|$ is at most $\frac{n-t}{\ell+1}$, since the component $c_H$ has at least $t$ 
vertices and each component in $\cC_1$ has $\ell+1$ vertices. 

We then have 
\begin{align}
 e-n&=\sum_{c\in\cC} \big(e(c)-n(c)\big)\notag\\
 &=e(c_H)-n(c_H) +\sum_{c\in\cC_1} (-1) + \sum_{c\in\cC_2}\big(e(c)-n(c)\big)\notag\\
 &\ge e(c_H)-n(c_H) - \frac{n-t}{\ell+1}.\label{eq-en1}
\end{align}
Define $\alpha:=\frac{e(c_H)}{m}$, and 
note that $\alpha\ge 1$. Recall from Lemma~\ref{lem-unifdens} that 
$\frac{e(c_H)}{n(c_H)}=\frac{m}{t}$, and hence $\frac{n(c_H)}{t}=\frac{e(c_H)}{m}=\alpha$. Hence 
$e(c_H)-n(c_H)=\alpha(m-t)\ge m-t$. Combining this with~\eqref{eq-n} and \eqref{eq-en1}, and using 
the notation $\Delta:=m-\ell$, we obtain
\begin{align}
e-n&\ge m-t - \frac{n-t}{\ell+1}\notag\\
&= m-t-\frac{1}{\ell+1}\left(\ell+1+\ell\cdot \frac{\ell+1 - t}{m-\ell}-t\right)\notag\\
&= \ell+\Delta-t-\frac{1}{\ell+1}\left(\ell+1+\ell\cdot \frac{\ell+1 - t}{\Delta}-t\right)\notag\\
&= \ell+\Delta -t -1 -\frac{\ell}{\Delta}+\frac{\ell t}{(\ell+1)\Delta}+\frac{t}{\ell+1}\notag\\
&=\Delta-1+\ell\cdot\frac{\Delta-1}{\Delta}-t\left(1-\frac{\ell}{(\ell+1)\Delta}-\frac{1}{\ell+1}
\right)\notag\\
&=\Delta-1+\ell\cdot\frac{\Delta-1}{\Delta}-t\cdot\frac{\ell(\Delta-1)}{(\ell+1)\Delta}\notag\\
&=\frac{\Delta-1}{\Delta}\left(\Delta+\ell-t\cdot\frac{\ell}{\ell+1}\right)\notag\\
&\ge 0,\label{eq-ne2}
\end{align}
where \eqref{eq-ne2} follows from
$\Delta+\ell-t\cdot\frac{\ell}{\ell+1}>\Delta+\ell -t=m-t>0$. 

We conclude that $e-n\ge 0$, showing that $\barG$ is not an exotic coBUG, a contradiction.
\end{proof}

Theorem~\ref{thm-mt} only applies to BUGs with betweenness smaller than~1. This bound is 
sharp in the sense that there is an infinite family of BUGs of betweenness 1 whose complements 
consist of a non-star uniform component combined with a suitable number of star 
components. This is shown by the next example.

\begin{example}\label{ex-mt}
Let $H$ be a complete multipartite graph with $t$ vertices and $m$ edges, and suppose that $m>t$. 
Recall from Observation~\ref{obs-multipar} that $H$ is $(m,t)$-uniform. Define $\ell:=m-1$. Let 
$\barG$ be the graph obtained as a disjoint union of $H$ with $m-t$ disjoint copies of 
$K_{1,\ell}$. We claim that $\barG$ is a coBUG whose complement is a BUG of betweenness~1. To see 
this, note first that $\barG$ has the same number of vertices and edges: indeed, the number of 
vertices is $n:=(m-t)(\ell+1)+t=(m-t)m+t=m^2-mt+t$, while the number of edges is 
$(m-t)(m-1)+m=m^2-mt+t$. It then remains to verify that all the vertices in $\barG$ have the same 
co-betweenness and apply Lemma~\ref{lem-density}. For a vertex $x$ belonging to a copy of 
$K_{1,\ell}$, the co-betweenness is
\begin{align*}
 coB(x)&=\frac{\ell}{n-(\ell+1)}=\frac{m-1}{m^2-mt+t-m}=\frac{1}{m-t},
\end{align*}
while for a vertex $y$ from $H$, Observation~\ref{obs-coB-mt-unif} yields
\begin{align*}
 coB(y)&=\frac{m}{n-t}=\frac{m}{m^2-mt+t-t}=\frac{1}{m-t}.
\end{align*}
Since all the vertices have the same co-betweenness, $\barG$ is a coBUG, and since $\barG$ has the 
same number of vertices and edges, its complement is a BUG of betweenness~1 by 
Lemma~\ref{lem-density}.
\end{example}

We are also able to construct an example of a sequence of coBUGs combining star components and 
non-star components, whose betweenness converges to 1 from above. 

\begin{example}\label{ex-above}
Fix a natural number $t\ge 2$, and define the following parameters:
\begin{align*}
k&=4t^2-4t-1\\
\ell&=4t^2-1\\
b&=t+1\\
c&=2t.
\end{align*}
Let $\barG$ be the disjoint union of $k$ copies of the star $K_{1,\ell}$ and $b$ copies of the 
complete bipartite graph $K_{c,c}$. We will show that $\barG$ is a coBUG with betweenness value 
$1+\frac{1}{4t}$. The number of vertices of $\barG$ is
\[
n:=k(\ell+1)+2bc=4t^2(4t^2-4t-1)+4t(t+1)=16t^4-16t^3+4t.
\]
A vertex $x$ in a component isomorphic to $K_{1,\ell}$ has co-betweenness
\[
coB(x)=\frac{\ell}{n-(\ell+1)}=\frac{4t^2-1}{16t^4-16t^3-4t^2+4t}=\frac{1}{4t^2-4t},
\]
while a vertex $y$ in a component isomorphic to $K_{c,c}$ has co-betweenness
\[
coB(y)=\frac{c^2}{n-2c}=\frac{4t^2}{16t^4-16t^3}=\frac{1}{4t^2-4t},
\]
showing that $\barG$ is indeed a coBUG. To determine the betweenness, we first compute the number 
of edges
\[
e:=k\ell+bc^2=16t^4-12t^3-4t^2+4t+1,
\]
and then use Lemma~\ref{lem-density} to determine that the betweenness is
\[
\frac{e}{n}=\frac{16t^4-12t^3-4t^2+4t+1}{16t^4-16t^3+4t}=\frac{4t+1}{4t}.
\]
\end{example}

By Lemma~\ref{lem-stars} and Theorem~\ref{thm-mt}, any exotic coBUG must consist of some number of 
star components isomorphic to $K_{1,\ell}$ together with some number of non-uniform components. 
Getting a valid coBUG by combining non-uniform components with star components is quite tricky, 
even when we do not require low betweenness. However, with the help of a computer-assisted search, 
we can show that, at least without the low betweenness restriction, such coBUGs exist. 

Our construction is based on the concept of inflation. For a graph $H$ and a number $k$, the 
\emph{$k$-fold inflation of $H$} is the graph $H[k]$ obtained by replacing each vertex $x$ of $H$ by 
a $k$-tuple of new vertices $x_1,\dotsc,x_k$ forming a clique, and replacing every edge $xy$ of $H$ 
by the $k^2$ edges $x_iy_j$ for all $1\le i,j\le k$. In particular, if $H$ has $n$ vertices and $e$ 
edges, then $H[k]$ has $kn$ vertices and $n\binom{k}{2}+ek^2$ edges. 

We can now present the promised example of a coBUG combining tree components with non-uniform 
components.

\begin{example} Let $H$ be the 27-fold inflation of the cycle of length 721. Note that $H$ is not 
a uniform graph. Although every vertex of $H$ has the same number of edges close to it (namely 
$4\cdot 27^2+3\cdot\binom{27}{2}=3969$), there are two types of edges: the first type are the edges 
connecting a pair of vertices $x_ix_j$ obtained by inflating a single vertex $x$ of the original 
cycle (which have $3\cdot 27=81$ vertices close to them), and the second type are the edges 
connecting a pair of the form $x_iy_j$ with $xy$ an edge of the original cycle (which have 
$4\cdot 27=108$ vertices close to them).

Let $\barG$ be the graph obtained as a disjoint union of $H$ and 81 copies of the star $K_{1,3924}$. 
In particular, $\barG$ has $n:=27\cdot 721+ 81\cdot 3925=337392$ vertices and $e:=721\cdot 
\binom{27}{2}+721\cdot27^2+81\cdot 3924=1096524$ edges. 

To verify that $\barG$ is a coBUG, we compute the co-betweenness of every vertex. For a vertex $w$ 
of a star component, we clearly have $coB(w)=\frac{3924}{n-3925}=\frac{3924}{333467}$. For a vertex 
$v$ of $H$, we first observe that an edge of the first type has weight 
$\frac{1}{n-81}=\frac{1}{337311}$, while an edge of the second type has weight 
$\frac{1}{n-108}=\frac{1}{337284}$. This gives
\[
coB(v)=\frac{3\cdot\binom{27}{2}}{337311}+\frac{4\cdot27^2}{337284}=\frac{3924}{333467},
\]
showing that all the vertices have the same co-betweenness, and $\barG$ is indeed a coBUG. It is not 
exotic, however, since the betweenness of the 
corresponding BUG $G$ is $B(G)=\frac{e}{n}=\frac{13}{4}\ge1$ by Lemma~\ref{lem-density}. 

Note that this construction also works when we replace the original cycle of length 721 by any 
graph on 721 vertices formed by a disjoint union of cycles of length at least~4.
\end{example}

\subsection{No exotic BUGs with betweenness up to \texorpdfstring{$\frac{9}{10}$}{9/10}}\label{ssec-small}

In this section, we show that there are no exotic BUGs whose betweenness value is at most 
$\frac{9}{10}$. To reach this result, we will show that any coBUG containing a connected 
component isomorphic to $K_{1,\ell}$ for some $\ell\le 8$ must only consist of a disjoint union of 
some number of copies of~$K_{1,\ell}$. Since any potential exotic coBUG must contain a star 
component, as well as at least one non-star component, this means that the star components in 
exotic coBUGs are of the form $K_{1,m}$ for some $m\ge 9$, implying that the coBUG has 
edge-to-vertex ratio greater than $\frac{m}{m+1}\ge\frac{9}{10}$. 

We begin by collecting general facts about the possible structure of a non-star component in an 
exotic coBUG.

\begin{lemm}\label{lem-cobug-comp-structure}
Let $\barG$ be an exotic coBUG on $n$ vertices, and let $H$ be a non-star component of~$\barG$. 
Then $H$ has the following properties.
\begin{enumerate}
    \item $H$ has no vertex of degree 1.
    \item $H$ has no vertex of degree 2 whose neighbours are adjacent; in other words, there 
is no triangle in $H$ containing a vertex of degree~2.
    \item Suppose that $x$ and $y$ are two adjacent vertices of $H$ which both have degree~2. 
Then $H$ does not contain any cycle of length 4 or 5 containing both $x$ and~$y$.
    \item $H$ has no triple of vertices $x,y,z$, all of degree~2, such that $xy$ and $yz$ are 
both edges of~$H$.
\end{enumerate}
\end{lemm}

\begin{proof}
The first claim follows from Lemma~\ref{lem-stars}. 

For the rest of the proof, recall from Theorem~\ref{thm-mt} that $H$ is not a uniform graph, and in 
particular $H$ is neither a cycle nor a complete multipartite graph.

To prove the second claim, let $x \in V(\barG)$ be a vertex of degree two adjacent to vertices $y$ 
and $z$ such that $yz \in E(\barG)$. Note that $\close(x)$ is a subset of $\close(y)$ as 
well as of $\close(z)$. In view of Lemma~\ref{lem-close-inclusion}, this means that 
$\close(x)=\close(y)=\close(z)$. Since $H$ is not a cycle, at least one of the two vertices $y,z$ 
must have degree greater than~2. Suppose that $\deg(y)>2$ and let $w$ be any vertex adjacent to $y$.

Since any edge incident to $w$ is close to $y$, it must be also close to $x$, and therefore any 
edge incident to $w$ must connect it to $y$ or to~$z$. Since $H$ has no vertices of degree 1, any 
vertex of $H$ different from $y$ or $z$ must have degree two and be adjacent to both $y$ 
and~$z$. It follows that $H$ is isomorphic to the complete multipartite graph $K_{1,1,\ell}$ for 
some $\ell\ge 1$, which is impossible since $H$ is not uniform. This proves the second claim.

To prove the third claim, let $x$ and $y$ be two adjacent vertices of degree 2 in $H$, let $u$ be 
the neighbor of $x$ different from $y$ and let $v$ be the neighbor of $y$ different from~$x$. 
From the second claim, it follows that $u\neq v$. To prove the third claim, we need to show that 
$u$ and $v$ are not adjacent and that they do not share any common neighbor.

Suppose first that $uv$ is an edge of~$H$. Then $\close(x)\subseteq \close(u)$ and 
$\close(y)\subseteq\close(v)$. By Lemma~\ref{lem-close-inclusion}, these inclusions must in fact be 
equalities. However, any edge incident to $u$ other than $ux$ and $uv$ would belong to 
$\close(v)\setminus\close(y)$, hence $\deg(u)=2$, and by the same argument, $\deg(v)=2$, showing 
that $H$ is isomorphic to $C_4$, which is impossible.

Now suppose that $u$ and $v$ have a common neighbor~$w$. Observe that each of the five edges $xy$, 
$yv$, $vw$, $wu$, and $ux$ has weight at least $\frac{1}{n-4}$, with equality if and only if both 
vertices of the edge have degree~2. Since $coB(x)=coB(w)$, it follows that the total weight of 
edges in $\close(x)\setminus\close(w)$ is the same as the total weight of edges in 
$\close(w)\setminus\close(x)$. The set $\close(x)\setminus\close(w)$ contains only the edge $xy$, 
which has weight $\frac{1}{n-4}$, while $\close(w)\setminus\close(x)$ contain the edge $wv$, as 
well as all the other edges incident to $w$ or $v$, other than $yv$ and $wu$. It follows that $wv$ 
has weight $\frac{1}{n-4}$, and $\close(w)\setminus\close(x)$ does not contain any other edge. 
Symmetrically, $\close(w)\setminus\close(y)$ only contains the edge $wu$ of weight $\frac{1}{n-4}$. 
It follows that the vertices $u$, $v$ and $w$ all have degree 2, and $H\isom C_5$, which is 
impossible. This proves the third claim.

To prove the fourth claim, note that since $H$ is not a cycle, it must contain a vertex $w$ of 
degree at least~3. We may assume without loss of generality that this vertex is adjacent to~$z$. 
From the third claim of the lemma, we know that $w$ is not adjacent to $x$ and has no common 
neighbors with $x$. The only edge in $\close(z)\setminus\close(w)$ is the edge $xy$, which has 
the smallest possible weight $\frac{1}{n-4}$. We will now show that there are at least two edges 
close to $w$ which are not close to~$z$. Let $u$ and $v$ be any two neighbors of $w$ different 
from~$z$. Note that $u$ is incident to at least one edge $e_u$ which does not connect it to $w$ or 
to $v$: indeed, if $u$ were only incident to the two edges $uw$ and $uv$, we would get a 
contradiction with the second claim. Similarly, $v$ is incident to an edge $e_v$ which does not 
connect it to $u$ or~$w$. The two edges $e_u$ and $e_v$ are close to $w$ but not to~$z$, and their 
combined weight is at least $\frac{2}{n-4}$. This shows that $coB(w)>coB(z)$, a contradiction.
\end{proof}

We will now deduce several simple but useful consequences of Lemma~\ref{lem-cobug-comp-structure}.

\begin{lemm}\label{lem-minimal-edge-weight}
    Let $H$ be a non-star component of an exotic coBUG $\barG$ on $n$ vertices. Then any edge 
of $H$ has at least four vertices close to it, and therefore its weight is at least $\frac{1}{n-4}$.
\end{lemm}
\begin{proof}
Let $e=xy$ be an edge of~$H$. By Lemma~\ref{lem-cobug-comp-structure}, both $x$ and $y$ have degree 
at least 2, and if both have degree exactly two, they do not share a common neighbor. It follows 
that $e$ is close to at least four vertices and has weight 
$w(e)=\frac{1}{n-|\close(e)|}\ge\frac{1}{n-4}$.
\end{proof}

\begin{lemm}\label{lem-non-uniform-size}
Any non-star component of an exotic coBUG has at least six vertices.
\end{lemm}
\begin{proof}
Let $H$ be a non-star component of an exotic coBUG on $n$ vertices, and suppose that $H$ has at 
most five vertices. Since $H$ is not uniform, and therefore in particular not complete multipartite, 
in contains an induced copy of $K_1+K_2$ by Observation~\ref{obs-multipar}. In other words, it 
contains three vertices $x,y,z$ where $yz$ is an edge of $H$ but $xy$ and $xz$ are not. Since $x$ 
has degree at least two by Lemma~\ref{lem-cobug-comp-structure}, it is adjacent to two other 
vertices $u$ and~$v$. It follows that $x$ has degree exactly two and $V(H)=\{x,y,z,u,v\}$. Each of 
the two vertices $u$ and $v$ must have at least one neighbor other than $x$. By the second claim of 
Lemma~\ref{lem-cobug-comp-structure}, $uv$ is not an edge of~$H$. 

If one of the two vertices $u$ and $v$ has degree 2, we will obtain a cycle of length 4 or 5 with 
two adjacent vertices of degree two, which is impossible. If both $u$ and $v$ have degree 3, then 
$H$ is isomorphic to the graph $K_5^-$, obtained from $K_5$ by the removal of one edge. 
Then $\close(x)=E(H)\setminus\{yz\}$, while $\close(y)=E(H)$, contradicting 
Lemma~\ref{lem-close-inclusion}.
\end{proof}

Many of our arguments are based on the notion of excess $\ex(H)$ of a graph~$H$. We now collect some 
of the salient facts related to this notion.

\begin{obs}\label{obs-excdeg}
For any graph $H=(V,E)$, its excess can be expressed as
\[
\ex(H)=|E|-|V|=\left(\frac12\sum_{x\in V}\deg(x)\right)-|V|=\frac{1}{2}\sum_{x\in 
V}\left(\deg(x)-2\right).
\]
In particular, if $H$ has minimum degree at least 3, then $\ex(H)\ge\frac{|V|}{2}$.
\end{obs}

\begin{obs}\label{obs-excn}
Let $\barG$ be an exotic coBUG on $n$ vertices with a non-star component $H$ on $t$ vertices and a 
star component $K_{1,\ell}$. Then $\barG$ must have negative excess, and in particular, it must have 
more than $\ex(H)$ components isomorphic to $K_{1,\ell}$. Since $t\ge 6$ by 
Lemma~\ref{lem-non-uniform-size}, we get
\[n\ge (\ex(H)+1)(\ell+1)+t\ge (\ex(H)+1)(\ell+1)+6.\]
\end{obs}

\begin{lemm}\label{lem-exc2}
Any non-star component $H$ of an exotic coBUG has excess at least~2. 
\end{lemm}
\begin{proof}
Since $H$ is not a star, it has non-negative excess. Since it is not uniform, and therefore not a 
cycle, it has excess at least~1. By Observation~\ref{obs-excdeg}, a connected graph with excess 
equal to 1 and minimum degree 2 either contains two vertices of degree 3, or one vertex of degree 4, 
while all the remaining vertices have degree 2. However, any such graph is in contradiction with 
Lemma~\ref{lem-cobug-comp-structure}, since it contains a vertex of degree 2 in a triangle, two 
adjacent vertices of degree 2 on a cycle of length 4 or 5, or three vertices of degree 2 forming a 
path.
\end{proof}
\begin{cor}\label{cor-excn}
Any exotic coBUG has at least three star components. In particular, with the notation of 
Observation~\ref{obs-excn}, we have $n\ge 3(\ell+1)+t\ge 3\ell+9$.
\end{cor}

Another useful fact is the following relation between the degree of a vertex and the number of edges close to it in a non-star component of a complement of low-betweenness BUG.
   \begin{lemm}\label{lem-close-2deg}
       If $H$ is a non-star component of an exotic coBUG $\barG$, then any vertex $v \in V(H)$ of 
degree $d$ has at least $2d$ edges close to it.
   \end{lemm}
   \begin{proof}
       Consider any $w \in N(v)$. We want to show that any $w \in N(v)$ contributes on average by at least two edges to $\close(v)$.
       Clearly, every $w$ contributes at least by the edge $vw$.
       By claim 1 of Lemma \ref{lem-cobug-comp-structure}, either $\deg(w) \geq 3$ or $w$ has a 
neighbour $x$ outside of $N(v)$. In the latter case, $wx \in \close(v)$ is the second edge which 
uniquely contributes to the cardinality of $\close(v)$.
       In the case of $\deg(w) \geq 3$, we have that $w$ is adjacent to at least two vertices different from $v$.
       If at least one of these vertices is outside of $N(v)$, then the corresponding edge uniquely contributes  one to $\close(v)$. Otherwise, each such edge contributes by one half to $|\close(v)|$ and as there are at least two such edges, the resulting contribution to $|\close(v)|$ for $w$ is at least two.
   \end{proof}

\begin{lemm}\label{lem-mincob}
In an exotic coBUG on $n$ vertices, every vertex has co-between\-ness at least $\frac{6}{n-4}$.
\end{lemm}
\begin{proof}
Let $H$ be a non-star component of an exotic coBUG $\barG$ on $n$ vertices. It follows from 
Lemma~\ref{lem-cobug-comp-structure} that every edge is close to at least four vertices, and 
therefore has weight at least $\frac{1}{n-4}$. Moreover, at least one vertex $x$ of $H$ has degree 
at least 3, and by Lemma~\ref{lem-close-2deg}, such a vertex is close to at least six edges. In 
particular the co-betweenness of such a vertex (and therefore of any vertex in $\barG$) is at least 
$\frac{6}{n-4}$.
\end{proof}
   
\subsubsection{Ruling out small stars in exotic coBUGs}

Here we show that stars $K_{1,\ell}$ for $\ell \in \{0, 1, 2, \dots, 8\}$ cannot appear as 
components of exotic coBUGs. We start with small stars and proceed to larger ones. The arguments 
become steadily more and more complicated as the size of the stars considered grows.

For small stars, Corollary~\ref{cor-excn} by itself is sufficient.
\begin{lemm}\label{lem-K04}
For $\ell\in\{0,1,2,3,4\}$, there is no exotic coBUG containing $K_{1,\ell}$ as component.
\end{lemm}
\begin{proof}
Suppose that an exotic coBUG on $n$ vertices contains a component $K_{1,\ell}$, for some $\ell\le 
4$. The vertices in such a star component have co-betweenness $\frac{\ell}{n-(\ell+1)}\le 
\frac{4}{n-5}$. At the same time, the vertices of $\barG$ have co-betweenness at least 
$\frac{6}{n-4}$ by Lemma~\ref{lem-mincob}. This yields
\begin{align*}
&&\frac{4}{n-5}&\ge \frac{6}{n-4}\\
\iff&&4(n-4)&\ge 6(n-5)\\
\iff&&7&\ge n.
\end{align*}
This is impossible by Corollary~\ref{cor-excn}.
\end{proof}

\begin{remark}
Let $K_4^-$ be the graph obtained by removing an edge from~$K_4$. We may easily verify that the 
disjoint union of $K_4^-$ with $K_{1,4}$ is a coBUG. It is not a counterexample to 
Lemma~\ref{lem-K04}, because its betweenness value is 1, and hence it is not exotic. This is in 
fact a special case of the construction described in Example~\ref{ex-mt}.
\end{remark}

The argument from the proof of Lemma~\ref{lem-K04} can be extended to $\ell=5$.

\begin{lemm}\label{lem-K5}
There is no exotic coBUG with a component isomorphic to $K_{1,5}$.
\end{lemm}
\begin{proof}
If such an exotic coBUG existed, its vertices would have a betweenness of $\frac{5}{n-6}$, and by 
Lemma~\ref{lem-mincob},
\[
\frac{5}{n-6}\ge\frac{6}{n-4},
\]
which gives $n\le16$, contradicting Corollary~\ref{cor-excn}.
\end{proof}

For $\ell\ge 6$, slightly more elaborate arguments are needed. For a graph $H$, we define the 
\emph{closeness of $H$}, denoted $c(H)$, as the maximum number of edges that are close to a single 
vertex, i.e., $c(H)=\max_{x\in V(H)} |\closep{H}(x)|$. We collect the main relevant properties of 
this parameter in our next lemma.

\begin{lemm}\label{lem-closeness}
 Let $H$ be a non-star component of an exotic coBUG on $n$ vertices which contains a star component 
isomorphic to~$K_{1,\ell}$. Let $c$ be the closeness of~$H$. Then 
\begin{enumerate}
 \item $H$ has maximum degree at most $\lfloor c/2\rfloor$,
 \item $c>\ell$, 
 \item $n\le \lfloor\frac{c(\ell+1)-4\ell}{c-\ell}\rfloor$.
\end{enumerate}
\end{lemm}
\begin{proof}
The first part follows from Lemma~\ref{lem-close-2deg}. To prove the second part, assume for 
contradiction that every vertex of $H$ has at most $\ell$ edges close to it. From the first part, 
we know that all the vertices of $H$ have degree at most $\ell/2$, which means that every edge of 
$H$ has at most $\ell$ vertices close to it, and therefore its weight is at most 
$\frac{1}{n-\ell}$. It follows that the co-betweenness of any vertex in $H$ is at most 
$\frac{\ell}{n-\ell}$. However, $\barG$ contains star components isomorphic to $K_{1,\ell}$, whose 
vertices have co-betweenness $\frac{\ell}{n-(\ell+1)}>\frac{\ell}{n-\ell}$, showing that $\barG$ is 
not a coBUG.

To prove the last part, fix a vertex $x$ of $H$ that has $c$ edges close to it. Since every edge 
has weight at least $\frac{1}{n-4}$ by Lemma~\ref{lem-minimal-edge-weight}, we conclude that 
$coB(x)$ is at least $\frac{c}{n-4}$. Simultaneously, all the vertices of $\barG$ have 
co-betweenness $\frac{\ell}{n-(\ell+1)}$, since $\barG$ contains $K_{1,\ell}$ as a component.
Combining these estimates, while recalling that $c>\ell$ from the second part of this lemma and 
that $n\ge3\ell+9$ from Corollary~\ref{cor-excn}, yields
\begin{align*}
 && \frac{c}{n-4}&\le\frac{\ell}{n-(\ell+1)}\\
 \iff&&  c(n-(\ell+1))&\le\ell(n-4)\\
 \iff&& n(c-\ell)&\le c(\ell+1)-4\ell\\
 \iff&& n&\le \frac{c(\ell+1)-4\ell}{c-\ell}.
\end{align*}
The lemma follows.
\end{proof}

\begin{cor}\label{cor-closeness}
Observe that for a fixed integer $\ell\ge0$, the expression $\frac{c(\ell+1)-4\ell}{c-\ell}$ is a 
decreasing function of $c$ for $c\in (\ell,+\infty)$. Thus, for $c$ an integer greater than $\ell$, 
the expression is maximised when $c=\ell+1$, and Lemma~\ref{lem-closeness} then yields $n\le 
(\ell+1)^2-4\ell$. 
\end{cor}

\begin{lemm}\label{lem-K6}
There is no exotic coBUG with a component isomorphic to $K_{1,6}$.
\end{lemm}
\begin{proof}
Corollary~\ref{cor-closeness} yields $n\le 25$, while Corollary~\ref{cor-excn} yields $n\ge 27$, a 
contradiction.
\end{proof}

For $\ell=7$, we again need a more involved argument.
\begin{lemm}\label{lem-K7}
There is no exotic coBUG with a component isomorphic to $K_{1,7}$.
\end{lemm}
\begin{proof}
Suppose that $\barG$ is such a coBUG on $n$ vertices, and $H$ its non-star component. 
Corollary~\ref{cor-excn} yields $n\ge 30$.  

Let $c$ be the closeness of $H$. Part 2 of
Lemma~\ref{lem-closeness} shows that $c\ge 8$. If $c\ge 9$, then part 3 of the same lemma implies 
that $n\le \frac{9\cdot 8-4\cdot 7}{2}=22$, a contradiction (note that we use here the monotonicity 
that we pointed out in Corollary~\ref{cor-closeness}). 

We therefore conclude that $c=8$, and Corollary~\ref{cor-closeness} yields the 
bound $n\le 36$. Combining this with the inequality $n\ge 8(\ex(H)+1)+6$ from 
Observation~\ref{obs-excn}, we get $\ex(H)\le 2$, and in view of Lemma~\ref{lem-exc2}, 
$\ex(H)=2$. 

By part 1 of Lemma~\ref{lem-closeness}, all the vertices of $H$ have degree at 
most~4. Therefore, every edge of $H$ is close to at most eight vertices and has weight at most 
$\frac{1}{n-8}$. Note that due to the presence of the $K_{1,7}$ component, all the vertices of 
$\barG$ have co-betweenness $\frac{7}{n-8}$. In particular, every vertex of $H$ must be close to at 
least seven edges, and if it is close to exactly seven edges, all of them must have weight 
$\frac{1}{n-8}$.

Note that $H$ must contain a vertex of degree 2: if all its vertices had degree at least 3, then 
from Lemma~\ref{lem-non-uniform-size} and Observation~\ref{obs-excdeg}, we would get $\ex(H)\ge 
3$, which we have already excluded. Let $x$ then be a vertex of $H$ of degree~2, and let $y$ and 
$z$ be its two neighbors. Note that each of the two edges $xy$ and $xz$ is close to at most 6 
vertices, therefore it has weight at most $\frac{1}{n-6}$. By the argument of the previous 
paragraph, this means that $x$ must be close to eight edges, otherwise its co-betweenness would be 
too low. Consequently, the two vertices $y$ and $z$ both have degree~4 and they are not adjacent to 
each other.

Since $\ex(H)=2$, Observation~\ref{obs-excdeg} implies that $y$ and $z$ are the only vertices of 
degree 4 in $H$, and all the remaining vertices have degree~2. Since $x$ was chosen as an arbitrary 
vertex of degree 2, it follows that all the vertices of degree 2 are adjacent to $y$ and~$z$, 
showing that $H$ is the complete bipartite graph $K_{2,4}$. This, however, is impossible, since 
$K_{2,4}$ is uniform and cannot be a component of an exotic coBUG.
\end{proof}

\begin{lemm}\label{lem-K8}
There is no exotic coBUG with a component isomorphic to $K_{1,8}$.
\end{lemm}
\begin{proof}
We again fix an exotic coBUG $\barG$ on $n$ vertices, with star-components isomorphic to $K_{1,8}$ 
and a non-star component~$H$. Corollary~\ref{cor-excn} yields $n\ge 33$. Let $c$ be the closeness 
of $H$, which is at least 9 by part 2 of Lemma~\ref{lem-closeness}, and in fact it is exactly 9, since for 
$c\ge 10$ the third part of the same lemma would imply $n\le 29$, while for $c=9$, the lemma gives 
$n\le 49$. It follows that the excess of $H$ is at most 3, by Observation~\ref{obs-excn}.

From $c=9$, we conclude that the maximum degree of $H$ is at most 4 by part 1 of Lemma~\ref{lem-closeness}, and therefore any edge of $H$ 
is close to at most eight vertices and has weight at most $\frac{1}{n-8}$. All the vertices of $H$ 
then must be close to exactly 9 edges, because a vertex close to at most eight edges would have 
co-betweenness at most $\frac{8}{n-8}$, less than the actual value of $\frac{8}{n-9}$.

Every vertex of $H$ therefore has degree at least 3, since a vertex of degree 2 would be close to 
at most eight edges. It follows by Observation~\ref{obs-excdeg} that $\ex(H)\ge |V(H)|/2$. This, in 
view of the bound $|V(H)|\ge 6$ from Lemma~\ref{lem-non-uniform-size} and the inequality $\ex(H)\le 
3$ derived above, tells us that $\ex(H)=3$, $|V(H)|=6$, and every vertex of $H$ has degree~3. 

Since $H$ is a 3-regular graph on six vertices, it has nine edges. Since each of its vertices is 
close to nine edges, it follows that every vertex is close to every edge, therefore $H$ is a 
complete multipartite graph by Observation~\ref{obs-multipar}, and therefore cannot be a component 
of a coBUG.
\end{proof}

We are finally ready to combine the ingredients into the proof of our main result.

\begin{proof}[Proof of Theorem~\ref{thm-main}]
For contradiction, let $G$ be an exotic BUG on $n$ vertices with betweenness at most 9/10, and let 
$\barG$ be its complement. By Lemma~\ref{lem-stars}, $\barG$ must contain a star component 
$K_{1,\ell}$ for some integer~$\ell$, and all its star components are of the same size. From the 
previous lemmas of this section, we know that $\ell\ge 9$. Furthermore, $\barG$ must also contain at 
least one non-star component, and all such non-star components have more edges than vertices by 
Lemma~\ref{lem-exc2}. Let $t$ be the total number of vertices in the non-star components of $\barG$, 
and let $m$ be their total number of edges; in particular, $m$ is greater than~$t$. Let $S$ be the 
number of star components in~$\barG$, and let $e$ be the number of edges of~$\barG$. The 
betweenness of $G$ is equal to $e/n$ by Lemma~\ref{lem-density}. We calculate:
\begin{align*}
 \frac{9}{10}\ge B(G)&=\frac{e}{n}\\
 &=\frac{\ell S+ m}{(\ell+1)S+t}\\
 &=\frac{\ell(S+(m/\ell))}{(\ell+1)(S+t/(\ell+1))}\\
 &> \frac{\ell}{\ell+1}\ge 9/10,
\end{align*}
a contradiction.
\end{proof}

\section{Conclusion and open questions}
Our results provide a complete characterization of betweenness-uniform graphs with betweenness at most 9/10, by showing that they are complements of graphs that consist of disjoint copies of $K_{1,\ell}$ for some integer~$\ell$, in which case their betweenness is equal to $\ell/(\ell+1)$. Our main open problem is to settle Conjecture~\ref{our-conj}, which states that this characterization can be extended to all BUGs of betweenness less than 1.

For BUGs of betweenness equal to 1, such simple characterization is no longer available, and indeed, we have provided several different constructions of BUGs that have betweenness equal to~1. It is an open problem to provide a full description of all the BUGs with betweenness equal to~1.

For betweenness values larger than 1, even less is known. In particular, we do not know which values $\alpha>1$ are equal to the betweenness of a BUG. Clearly, any such $\alpha$ must be rational, but are there any other restrictions? Is there a rational number greater than 1 that is not equal to the betweenness of any BUG? Is the set of such rational numbers infinite? Is it unbounded? Is it even dense inside the interval $(1,+\infty)$? All these questions, and many others like them, remain open.

\section*{Acknowledgements}

DH and AP were supported by the Czech Science Foundation Grant No. 23-07074S.
VJ was supported by the Czech Science Foundation Grant No. 23-04949X.
R\v{S} was partially supported by grant 22-17398S of the Czech Science Foundation. This project has also received partial funding from the European Research Council (ERC) under the European Union's Horizon 2020 research and innovation programme (grant agreement No 810115). This research was started during workshop KAMAK 2022, we are grateful to its organizers.

\bibliographystyle{abbrv}
\bibliography{bibliography.bib}

\end{document}